\documentclass[12pt]{amsart}
\usepackage{amssymb}
\usepackage{amsbsy}
\usepackage{amscd}

\theoremstyle{plain}
\newtheorem{thm}{Theorem}[section]
\newtheorem{prop}[thm]{Proposition}
\newtheorem{lem}[thm]{Lemma}

\theoremstyle{definition}
\newtheorem{dfn}[thm]{Definition}
\newtheorem{ex}[thm]{Example}
\newtheorem{rem}[thm]{Remark}

\numberwithin{equation}{section}

\newcommand{\G}{\Gamma}
\newcommand{\g}{\gamma}

\newcommand{\go}{\omega}

\newcommand{\gD}{\Delta}
\newcommand{\sm}{\left(\begin{smallmatrix}}
\newcommand{\esm}{\end{smallmatrix}\right)}
\newcommand{\mm}{\begin{pmatrix}}
\newcommand{\emm}{\end{pmatrix}}

\newcommand{\wt}{\widetilde}
\newcommand{\wh}{\widehat}

\newcommand{\ga}{\alpha}
\newcommand{\gb}{\beta}

\newcommand{\mbb}{\mathbb}
\newcommand{\mcal}{\mathcal}

\newcommand{\gl}{\lambda}
\newcommand{\gz}{\zeta}
\newcommand{\gL}{\Lambda}
\newcommand{\mU}{\mathcal U}
\newcommand{\mV}{\mathcal V}

\newcommand{\bi}{\binom}

\newcommand{\mF}{\mathcal F}
\newcommand{\mH}{\mathcal H}

\newcommand{\mJ}{\mathcal J}

\newcommand{\fJ}{\mathfrak J}
\newcommand{\fK}{\mathfrak K}

\newcommand{\mG}{\mathcal G}

\newcommand{\mQ}{\mathcal Q}

\newcommand{\bC}{\mathbb C}
\newcommand{\bR}{\mathbb R}
\newcommand{\bZ}{\mathbb Z}

\newcommand{\md}{\operatorname{\text{$\|$}}}

\newcommand{\ww}{\widehat{u}}
\newcommand{\wv}{\widehat{v}}
\newcommand{\mW}{\mathcal W}
\newcommand{\mmV}{\mathcal V}
\newcommand{\wJ}{\widehat{\mathcal J}}

\begin{document}

\title{Symmetric tensor representations,
quasimodular forms, and weak Jacobi forms}


\author[YoungJu Choie]{YoungJu Choie$^\ast$}

\address{Department of Mathematics and
PMI (Pohang Mathematical Institute), POSTECH, Pohang 790-784,
Korea}

\email{yjc@postech.ac.kr}

\author[Min Ho Lee]{Min Ho Lee$^\dagger$}

\address{Department of Mathematics, University of Northern Iowa,
Cedar Falls, IA 50614, U.S.A.}

\email{lee@math.uni.edu}

\thanks{$^\ast$Supported in part by NRF 2009008-3919 and NRF by the Ministry of Education, Science and
Technology 2009-0094069}
\thanks{$^\dagger$Supported in part by a summer fellowship from the
  University of Northern Iowa}


\begin{abstract}
We establish a correspondence between vector-valued modular forms
with respect to a symmetric tensor representation and quasimodular
forms.  This is carried out by first obtaining an explicit isomorphism
between the space of vector-valued modular forms with respect to a
symmetric tensor representation and the space of finite sequences of
modular forms of certain type.  This isomorphism uses Rankin-Cohen brackets
and extends a result of Kuga and Shimura, who considered the case of
vector-valued modular forms of weight two.  We also obtain a correspondence
between such vector-valued modular forms and weak Jacobi forms.
\end{abstract}


\subjclass[2000]{11F11, 11F50, 11F25,}

\maketitle \pagestyle{plain}

\section{\bf{Introduction and statement of main results}}
\label{S:in}

It is well-known that the integrals of certain vector-valued
differential forms attached to automorphic forms with respect to a
Fuchsian group play an important role in the arithmetic theory of
modular correspondences, and the study of such integrals was first
introduced by Eichler \cite{Ei57} and Shimura \cite{Sh59}.

In \cite{KS60} Kuga and Shimura studied a correspondence between
vector-valued differential forms, or equivalently vector-valued modular
forms of weight two, and certain finite sequences of scalar-valued modular
forms. One of our goals in this paper is to extend such a correspondence to
the case of vector-valued modular forms of arbitrary weight.  We then use
this correspondence to obtain an isomorphism between the space of
vector-valued modular forms with respect to a symmetric tensor
representation and that of quasimodular forms.  Quasimodular forms
generalize classical modular forms and were first introduced by Kaneko and
Zagier in \cite{KZ95}.  It appears naturally in various places (see
\cite{EO1, EO2, LR}, for instance).  In fact, an explicit isomorphism can
be obtained between the space of finite sequences of modular forms of
certain type and the space of vector valued modular form with respect to a
symmetric tensor representation.  This isomorphism is defined by
using Rankin-Cohen brackets for vector-valued modular forms. On the other
hand, finite sequences of modular forms of the same type are known to
correspond to certain quasimodular forms (see \cite{L09b}) and also to weak
Jacobi forms (cf .\ \cite{EZ85}).

In this paper we establish mutual correspondences among
vector-valued modular forms with respect to a symmetric tensor
representation, finite sequences of modular forms, quasimodular
forms, and weak Jacobi forms.

Given a discrete subgroup $\G$ of $SL(2, \bR)$ and integers $k$
and $n$ with $n \geq 0$, let $M_k (\G)$ and $\wh{M}^{n+1}_k (\G,
\rho_n)$ be the space of modular forms of weight $k$ for $\G$ and
the space of vector-valued modular forms of weight $k$ for $\G$
with respect to the $n$-th symmetric tensor representation
$\rho_n$, respectively.  We also denote by $QP_{k}^{m} (\G)$ and
$QM_{k}^m (\G)$ with $m \geq 0$ the spaces of quasimodular
polynomials of degree at most $m$ and quasimodular forms depth at
most $m$, respectively, of weight $k$ for $\G$ (see Section
\ref{S:qmp}). Our main results involve establishing certain
isomorphisms contained in the following theorems, which will be
proved in Section \ref{S:pm}

\begin{thm}\label{main}
(i) Given an integer $\ell$ with $0 \leq \ell\leq n$, the formula
\begin{equation} \label{E:wa}
\mmV_{k, n, \ell}(g_{\ell}) = [g_{\ell}, \wv_n]_{n-\ell}^{(k-n
+2\ell, -n)}
\end{equation}
for $g_\ell \in M_{k-n+2\ell}(\G)$ determines a complex linear map
\begin{equation} \label{E:2p}
\mmV_{k, n, \ell}: M_{k-n+2\ell}(\G) \to \wh{M}_k^{n+1}(\G,
\rho_n) ,
\end{equation}
where $\wv_n (z) = {^t}(z^n, \ldots, z, 1) \in
\bC^{n+1}$, and $[g_{\ell}, \wv_n]_{n-\ell}^{(k-n+2\ell, -n)}$
denotes the $(n-\ell)$-th Rankin-Cohen bracket of $g_{\ell}$ and
$\wv_n$ given by \eqref{E:he}.

(ii) If $\wh{\mG}_{k-n+2\ell}^{n+1} (\G,\rho_n) = {\rm Im}\,
(\mmV_{k, n, \ell})$ denotes the image of the map \eqref{E:2p} for
$0 \leq \ell\leq n$, then $\wh{M}_k^{n+1} (\G,\rho_n)$ has a
direct sum decomposition of the form
\[ \wh{M}_k^{n+1}(\G,\rho_n) = \bigoplus^n_{\ell =0}
\wh{\mG}_{k-n+2\ell}^{n+1} (\G,\rho_n) .\]%

(iii) The map \eqref{E:2p} determines an isomorphism
\[ M_{k-n+2\ell}(\G) \cong \wh{\mG}_{k-n+2\ell}^{n+1} (\G,\rho_n)
\]%
for each $\ell \in \{ 0,1, \ldots, n \}$; hence we obtain a
canonical isomorphism \begin{equation} \label{E:m}
\wh{M}_k^{n+1}(\G,\rho_n) \cong \bigoplus^n_{\ell =0} M_{k-n
+2\ell}(\G) \end{equation}
 between vector-valued modular forms
with respect to $\rho_n$ and finite sequences of modular forms.
\end{thm}

We can also consider an explicit linear map from vector-valued modular
forms to scalar-valued modular forms in \eqref{E:2p} by using Rankin-Cohen
brackets for vector-valued modular forms as in the next theorem.

\begin{thm}\label{main2}
Given integers $k$ and $n$ with $n \geq 0$, the formula
\[ \mW_{k,n} (F) =
([[\ww_n, F]]_0^{(-n,k)}, [[\ww_n, F]]_1^{(-n,k)}, \ldots,
[[\ww_n, F]]_n^{(-n,k)}) \]%
for $F \in \wh{M}^{n+1}_k (\G, \rho_n)$ determines a complex
linear map
\begin{equation} \label{E:7t}
\mW_{k,n}: \wh{M}^{n+1}_k (\G, \rho_n) \to \bigoplus^n_{\ell =0}
M_{k-n +2\ell}(\G),
\end{equation}
where the brackets are the Rankin-Cohen brackets in \eqref{E:hh}
and $\ww_n= {^t} (1, (-z), .., (-z)^n).$
\end{thm}

Given an element $(f_0, f_1, \ldots, f_n) \in \bigoplus^n_{\ell
=0} M_{k-n +2\ell}(\G)$, the corresponding polynomial
$\sum^n_{r=0} f_r (z) X^r$ belongs to $MP^n_{k-n} (\G)$, and this
correspondence determines an isomorphism
\[ \bigoplus^n_{\ell =0} M_{k-n +2\ell}(\G) \cong MP^n_{k-n} (\G) .\]
On the other hand, it is also known that
\[ QP^n_j (\G) \cong MP^n_{j -2n} (\G) \]
for $j >2n$ (see Section \ref{S:qmp}).  From these isomorphisms and
\eqref{E:m} we obtain the following result.

\begin{thm}\label{main00}
There is a canonical isomorphism.
\begin{equation} \label{E:6w}
\wh{M}^{n+1}_k (\G, \rho_n) \cong QP^n_{k+n} (\G)
\end{equation}
between vector-valued modular forms and quasimodular polynomials for $k
>n$.
\end{thm}

The next theorem gives an expression of the formula for the vector-valued
function corresponding to a quasimodular form via the isomrphism
\eqref{E:6w} in terms of derivatives of $\wv_n$ and the coefficients of
the given quasimodular form.

\begin{thm}\label{main0}
Let $\mF$ be the complex algebra of holomorphic functions on the Poincar\'e
upper half plane $\mH$.  If $F (z,X)$ is a polynomial of degree at most $n$
over the ring of holomorphic functions on $\mH$ given by
\[ F(z,X) = \sum^n_{r=0} f_r (z) X^r ,\]
we set
\begin{equation} \label{E:9i}
\mU_{n} (F(z,X))= \sum^n_{\ell=0} (-1)^{n} (n -\ell)! D^{\ell} (\wv_n)
f_{\ell} ,
\end{equation}
where $\wv_n$ is as in Theorem \ref{main} and $D = d/dz$.  Then this
formula determines an isomorphism
\[ \mU_{n}: QP_{k+n}^n(\G) \xrightarrow{\approx} \wh{M}_k^{n+1}(\G, \rho_n) \]
of complex vector spaces for each $k > n$.
\end{thm}

\begin{ex}
We consider the weight 12 cusp form
\begin{equation} \label{E:j4}
\gD (z) =q \prod_{n=1}^{\infty} (1 -q^n)^{24}
\end{equation}
for $\G (1) = SL(2, \bZ)$ with $q = e^{2\pi i z}$.
Using \eqref{E:wa} for $n =2$, $k =14$, $\ell =0$ and $g_0 = \gD$, we see that
\begin{equation} \label{E:2w}
\mmV_{14, 2, 0} (\gD) (z) = \gD'' (z) \sm z^2 \\ z \\ 1\esm
+13 \gD' (z) \sm 2z\\1\\0\esm  + 78 \gD (z) \sm 2 \\ 0 \\ 0 \esm \in \wh{M}_{14}^{3}(\G(1), \rho_2) .
\end{equation}
On the other hand, for $n=2$ the formula \eqref{E:9i}
determines a vector-valued modular form
\begin{equation} \label{E:3w}
\mU_2 (F (z,X)) = 2 \wv_2 (z) f_0 (z) + (D \wv_2) (z) f_1 (z)
+ (D^2 \wv_2) (z) f_2 (z) \in \wh{M}_{14}^{3}(\G(1), \rho_2)
\end{equation}
corresponding to a quasimodular polynomial $F(z,X) \in QP^2_{16}
(\G)$. By comparing \eqref{E:2w} and \eqref{E:3w} we obtain the
quasimodular polynomial
\[ (\mU_{2}^{-1} \circ \mmV_{14, 2, 0}) (\gD) (z,X)
= \frac 12 \gD'' (z) + 13 \gD' (z) X + 78 \gD (z) X^2 \in QP^2_{16} (\G) .\]
\end{ex}

\bigskip

\section{\bf{Vector-valued modular forms}}

In this section we review the definition of vector-valued modular
forms and introduce Rankin-Cohen brackets for such modular forms.

Let $\mH$ be the Poincar\'e upper half plane on which the group
$SL(2,\bR)$ acts as usual by linear fractional transformations.
Thus, if $z \in \mH$ and $\g = \sm a&b\\
c&d \esm \in SL(2,\bR)$, we have
\[ \g z = \frac {az+b} {cz+d}. \]%
For such $z$ and $\g$, we set
\begin{equation} \label{E:mz}
\fJ(\g,z) = cz+d , \quad \fK (\g,z) = \frac c{cz+d}.
\end{equation}
Then the resulting maps $\fJ, \fK: SL(2,\bR) \times \mH \to \bC$
satisfy
\[ \fJ(\g \g', z) = \fJ (\g, \g' z) \fJ (\g', z) ,\]
\[ \fK (\g, \g' z) = \fJ (\g', z)^2 (\fK (\g \g',z) - \fK (\g',z)) \]
for all $z \in \mH$ and $\g, \g' \in SL(2,\bR)$.

Let $\mF$ be the ring of holomorphic functions on $\mH$, and let
$\wh{\mF}^n$ with $n >0$ be the space of $\bC^n$-valued
holomorphic functions on $\mcal H$.  Throughout this paper, we
shall use ${^t}(\cdot)$ to denote the transpose of the matrix
$(\cdot)$. In particular, if $x_1, \ldots, x_\ell \in \bC$, the
corresponding column vector belonging to $\bC^\ell$ will be
denoted by ${^t}(x_1, \ldots, x_\ell)$.  Let $\rho: \G \to GL(n,
\bR)$ be a representation of a discrete subgroup $\G$ of $SL(2,\bR)$, and
let
\[ \rho^*: \G \to GL(n,\bR) \]
be the contragredient, or the inverse transpose, of $\rho$, so that
\[ \rho^* (\g) = {^t}\rho (\g)^{-1} \]%
for all $\g \in SL(2, \bR)$. If $f \in \mF$, $\wh{f} \in
\wh{\mF}^n$, $\g \in SL(2,\mbb R)$ and $\ell \in \mbb Z$, then we
set
\[ (f \mid_\ell \g) (z) = \fJ (\g,z)^{-\ell} f (\g z), \quad (\wh{f}
\mid_\ell \g) (z) = \fJ (\g,z)^{-\ell} \wh{f} (\g z) \]%
for all $z \in \mcal H$.  We now modify the usual definition of
modular forms by suppressing the growth condition at the cusps.

\begin{dfn} (i) Given an integer $k$, a {\em modular form of weight $k$
for $\G$\/} is an element $f \in \mF$ satisfying
\[ f \mid_k \g = f \]
for all $\g \in \G$.

(ii) An element $\wh{f} \in \wh{\mF}^n$ is a {\em vector-valued
modular form of weight $k$ for $\G$ with respect to $\rho$\/} if
it satisfies
\[ \wh{f} \mid_k \g = \rho(\g) \wh{f} \]
for all $\g \in \G$.
\end{dfn}

We shall denote by $M_k (\G)$ and $\wh{M}^n_k (\G, \rho)$ the
spaces of modular forms for $\G$ and vector-valued modular forms
for $\G$ with respect to $\rho$, respectively, of weight $k$.

Rankin-Cohen brackets are bilinear maps of scalar-valued modular forms (see
e.g. \cite{CM97}), and they can be extended to the vector-valued case as
described in the following lemma.

\begin{lem} \label{L:js}
(i) Given nonnegative integers $\gl_1, \gl_2, w$ and vector-valued
modular forms $\phi_1 \in \wh{M}^{n_1}_{\gl_1} (\G, \rho_1)$ and
$\phi_2 \in \wh{M}^{n_2}_{\gl_2} (\G, \rho_2)$, we set
\begin{equation} \label{E:w8}
[ \phi_1, \phi_2 ]^{(\gl_1, \gl_2)}_w  (z)= \sum^w_{r=0} (-1)^r
\bi {\gl_1 +w-1} {w-r} \bi {\gl_2 +w-1} {r} \phi^{(r)}_1 (z)
\otimes \phi^{(w-r)}_2 (z)
\end{equation}
for all $z \in \mH$.  Then $[ \phi_1, \phi_2 ]^{(\gl_1, \gl_2)}_w$
is a vector-valued modular form belonging to $M^{n_1 n_2}_{\gl_1
+\gl_2 +2w} (\G, \rho_1 \otimes \rho_2)$.

(ii) Given vector-valued modular forms $\phi \in \wh{M}^n_\ga (\G,
\rho)$ and $\psi \in \wh{M}^n_\gb (\G, \rho^*)$ with $\ga, \gb
\geq 0$ and a nonnegative integer $w$, we set
\begin{equation} \label{E:hh}
[[ \phi, \psi ]]^{(\ga, \gb)}_w  (z)= \sum^w_{r=0} (-1)^r \bi {\ga
+w-1} {w-r} \bi {\gb +w-1} {r} {^t}(\phi^{(r)} (z)) \psi^{(w-r)}
(z)
\end{equation}
for all $z \in \mH$, where $\rho^*$ is the contragredient of $\rho$.
Then $[[ \phi, \psi ]]^{(k, \ell)}_w$ is a
modular form belonging to $M_{k + \ell +2w} (\G)$.
\end{lem}

\begin{proof}
See \cite{L07x}.
\end{proof}

From Lemma \ref{L:js} we obtain the bilinear maps
\[ [ \; , \; ]^{(\gl_1, \gl_2)}_w: \wh{M}^{n_1}_{\gl_1}
(\G, \rho_1) \times \wh{M}^{n_2}_{\gl_2} (\G, \rho_2) \to M^{n_1
n_2}_{\gl_1 +\gl_2 +2w} (\G, \rho_1 \otimes \rho_2) ,\]
\begin{equation} \label{E:bi}
[[ \;, \; ]]^{(\ga, \gb)}_w: \wh{M}^n_k (\G, \rho)
\times \wh{M}^n_\ell (\G, \rho^*) \to M_{k + \ell +2w} (\G) ,
\end{equation}
which may be regarded as Rankin-Cohen brackets for vector-valued
modular forms.    We can consider the bracket given by
\eqref{E:w8} when one of the modular forms is scalar valued by
considering it as a 1-dimensional vector-valued modular form with
respect to the trivial representation.  Thus, for example, we have
a bilinear map of the form
\[ [ \; , \; ]^{(\gl, \mu)}_w: M_{\gl}
(\G) \times \wh{M}^{n}_{\mu} (\G, \rho) \to M^{n}_{\gl +\mu +2w}
(\G, \rho) ,\]%
given by
\begin{equation} \label{E:he}
[ f, \phi ]^{(\gl, \mu)}_w  = \sum^w_{r=0} (-1)^r \bi {\gl +w-1}
{w-r} \bi {\mu +w-1} {r} f^{(r)} \phi^{(w-r)}
\end{equation}
for $f \in M_{\gl} (\G)$ and $\phi \in \wh{M}^{n}_{\mu} (\G,
\rho)$.

\begin{rem}Although the proof of Lemma \ref{L:js} in \cite{L07x} was
given for vector-valued modular forms of nonnegative weight, we
shall use the same formulas to consider brackets for modular forms
of negative weight by using
\[ \bi kr = \frac {k (k-1) \cdots (k-r+1)} {r!} \]
for $k,r \in \bZ$ with $r>0$.  In particular, the bracket in the
statement of Theorem \ref{main} is the extended version of the
Rankin-Cohen bracket in \eqref{E:he}.  Similarly, the bracket in
\eqref{E:hh} can also be extended to the negative weight case, and
such brackets are used in Theorem \ref{main2}.
\end{rem}

\bigskip

\section{\bf{Symmetric tensor representations}}

In this section we introduce certain vector-valued modular forms
with respect to symmetric tensor representations and discuss some
of their properties.

We fix a positive integer $n$, and denote by
\[ \rho_n: SL(2,\bR) \to GL(n+1,\bR) \]
the $n$-th symmetric tensor representation of $SL(2,\bR)$ given by
\begin{equation} \label{E:jp}
\rho_n (\g) \bi {z_1} {z_2}^n = \biggl( \g \bi {z_1} {z_2}
\biggr)^n
\end{equation}
for all $\g \in SL(2,\bR)$, where
\[ \bi {z_1} {z_2}^n = {^t}(z_1^n, z_1^{n-1} z_2, \ldots, z_1
z_2^{n-1}, z_2^n) \in \bC^{n +1} \]%
with $\sm z_1\\ z_2 \esm \in \bC^2$.

\begin{prop}\label{syme}
Let $\G$ be a discrete subgroup of $SL(2, \bR)$, and let $\wv_n ,
\ww_n \in \wh{\mF}^n$ be vector-valued functions on $\mH$ given by
\begin{equation} \label{E:mv}
\wv_n (z) = \bi {z}1^n ,\quad \ww_n (z) = \bi 1{-z}^n
\end{equation}
for all $z \in \bC$.  Then $\wv_n$ and $\ww_n$ are vector-valued
modular forms with
\[ \wv_n \in \wh{M}^{n+1}_{-n} (\G, \rho_n), \quad \ww_n \in
\wh{M}^{n+1}_{-n} (\G, \rho^*_n) ,\]%
where $\rho^*_n: SL(2,\bR) \to GL(n+1,\bR)$ is the contragredient
of $\rho_n$.
\end{prop}

\begin{proof}
For  $\g = \sm a&b\\ c&d \esm \in SL(2,\bR)$, using \eqref{E:jp},
we see that
\begin{align} \label{E:x4}
\rho_n (\g) \wv_n (z) &= \rho_n (\g) \bi {z}1^n
= \biggl( \g \bi {z}1 \biggr)^n\\
&= \bi {az+b} {cz+d}^n =(cz+d)^n \bi {\g z}1 \notag\\
&= \fJ (\g,z)^{n} \wv_n (\g z). \notag
\end{align}
for all $z \in \mH$; hence we obtain $\wv_n \in \wh{M}^{n+1}_{-n}
(\G, \rho_n)$.  On the other hand, we have
\begin{align*}
\rho^*_n (\g) \ww_n (z) &= \rho_n ({^t}\g^{-1}) \ww_n (z) =
\biggl( {^t}\g^{-1} \bi 1{-z}
\biggr)^n \\
&= \biggl( \mm d& -c\\ -b&a \emm \bi 1{-z} \biggr)^n
= \bi {cz+d} {-(az+b)}^n\\
&= (cz+d)^n \bi 1{-(\g z)}^n\\
&= \fJ (\g,z)^n \ww_n (\g z),
\end{align*}
which show that $\ww_n \in \wh{M}^{n+1}_{-n} (\G, \rho^*_n)$.
\end{proof}

We now introduce the matrix-valued function
\[ L_n: \mH \to GL (n+1, \bC) \]
defined by
\[ L_n (z) = \rho_n \mm 1&z \\ 0&1  \emm \]
for all $z \in \mH$.  Then for $\g = \sm a&b \\ c&d \esm \in SL(2,
\bR)$ it is known (see p.$266$ in \cite{KS60}) that
\begin{equation} \label{E:pd}
L_n (\g z)^{-1} \rho_n (\g) L_n (z) = \rho_n \mm \fJ &0 \\ c& \fJ
\emm = \fJ^n A,
\end{equation}
where $A$ is the $(n+1) \times (n+1)$ lower triangular matrix
whose $(r+1)$-th row is of the form
\[ \Bigl( c^r \fJ^{n-r}, \sm r\\1 \esm c^{r-1} \fJ^{n-r -1}, \ldots,
\sm r\\r-1 \esm c \fJ^{n-2r +1}, \fJ^{n-2r}, 0, \ldots, 0 \Bigr) \]%
for $0 \leq r \leq n$ with $\fJ = \fJ (\g,z)$ in \eqref{E:mz}.

\begin{lem} \label{KS-2}
Let $\go \in \wh{M}^{n+1}_k (\G, \rho_n)$ be a vector-valued
modular form of the form
\begin{equation} \label{E:gs}
\go (z) = L_n (z) \cdot {^t}(f_0 (z), f_1 (z), \ldots, f_n (z))
\end{equation}
for all $z \in \mH$, where $f_0, \ldots, f_n \in \mF$.  If $r$ is an
integer with $0 \leq r \leq n$ such that $f_\ell \equiv 0$ for all $\ell <
r$, then $f_r$ is a modular form belonging to $M_{k -n+2r} (\G)$.
\end{lem}

\begin{proof}
Given $\go \in \wh{M}^{n+1}_k (\G, \rho_n)$ as in \eqref{E:gs} and
an element $\g \in \G$, using \eqref{E:pd} and the relation $\go
\mid_k \g = \rho (\g) \go$, we see that
\begin{align*}
\fJ (\g,z)^{-k} L_n (\g z) &\cdot {^t}(f_0 (\g z), f_1 (\g z),
\ldots, f_n (\g z))\\
&= \rho_n (\g) L_n (z) \cdot {^t}(f_0 (z), f_1 (z), \ldots, f_n
(z))\\
&= L_n (\g z) \fJ (\g,z)^n A \cdot {^t}(f_0 (z), f_1 (z), \ldots,
f_n (z));
\end{align*}
hence we have the identity
\begin{align*}
\fJ (\g,z)^{-k} &\cdot {^t}(f_0 (\g z), f_1 (\g z),
\ldots, f_n (\g z))\\
&= \fJ (\g,z)^n A \cdot {^t}(f_0 (z), f_1 (z), \ldots, f_n (z));
\end{align*}
of vectors in $\bC^{n+1}$.  Thus by comparing the $(r+1)$-th entries
we obtain
\[ \fJ (\g,z)^{-k} f_r (\g z) =  \fJ (\g,z)^{n-2r} f_r(z) ,\]
and therefore the lemma follows.
\end{proof}

If $\gl, w, r \in \bZ$, we set
\begin{equation} \label{E:j9}
\ga^\gl_{w,r} = (-1)^r \bi {\gl +w -1} {w -r} \bi {w -n -1} {r} .
\end{equation}
Then from \eqref{E:he} we see that
\begin{equation} \label{E:j4}
[f, \wv_n]_{w}^{(\gl, -n)} = \sum_{r=0}^{w}
\alpha^\gl_{w, r} f^{(r)} \wv_n^{(w-r)}
\end{equation}
for $f \in M_{\gl} (\G)$.

\begin{lem}\label{KS-1}
If $g_j\in M_{k-n+2j}(\G)$ with $0\leq j\leq 2n$, we have
\[ [g_j(z), \wv_n (z)]_{n-j}^{(k-n+2j, -n)} = L_n (z) \cdot
{^t} (g^L_0 (z), g^L_1 (z), \ldots, g^L_n (z)) \]%
for all $z \in \mH$, where
\[ g^L_\ell = \begin{cases}
0 & \text{for $0 \leq \ell \leq j-1$;}\\
(n-\ell)! \alpha^{k-n+2j}_{n-j, \ell-j} D^{\ell -j} g_j & \text{for $j
\leq \ell \leq n$}
\end{cases}\]%
with $D =d/dz$ and $\ga_{k -n +2j, \ell -j}$ is as in \eqref{E:j9}.
\end{lem}

\begin{proof}
A direct computation shows that
\[ (D^n \wv_n, D^{n-1} \wv_n, \ldots, \wv_n) = L_n (z) \cdot {\rm
diag} \, [n!, (n-1)!, \ldots, 1] ,\]%
where ${\rm diag}\, [\cdots]$ denotes the $(n+1) \times (n+1)$
diagonal matrix having $[\cdots]$ as its diagonal entries.  Thus
we see that
\begin{align*}
L_n (z) \cdot {^t} (g^L_0 (z)&, g^L_1 (z), \ldots, g^L_n (z))\\
& = (D^n \wv_n, D^{n-1} \wv_n, \ldots, \wv_n) \cdot
{\rm diag} \, [n!, (n-1)!, \ldots, 1]^{-1}\\
&\hspace{.8in} \times  {^t} (g^L_0 (z), g^L_1 (z), \ldots, g^L_n
(z))\\
&= \sum^n_{i=0} \frac 1{(n-i)!} (D^{n-i} \wv_n (z)) g_i^L (z)\\
&= \sum^n_{i=j} \alpha^{k-n+2j}_{n-j, i-j} (D^{n-i}
\wv_n (z)) (D^{i-j} g_j z)\\
&= \sum_{\ell=0}^{n-j} \alpha^{k-n+2j}_{n-j, \ell} (D^{n-\ell-j} \wv_n (z))
(D^{\ell} g_j(z))
\end{align*}
for all $z \in \mH$.  On the other hand, from \eqref{E:j4} we see that
\[ [g_j, \wv_n]_{n-j}^{(k-n+2j,-n)} = \sum_{\ell=0}^{n-j}
\alpha^{k-n+2j}_{n-j, \ell} (D^{\ell}g_j) (D^{n-j-\ell} \wv_n)  ;\]
hence the lemma follows.
\end{proof}

\bigskip

\section{\bf{Quasimodular and modular polynomials}} \label{S:qmp}

In this section we describe a correspondence between quasimodular
and modular polynomials as well as the identification of
quasimodular forms with quasimodular polynomials.  By combining
these isomorphisms with the main theorems stated in Section
\ref{S:in} we establish a correspondence between vector-valued
modular forms and quasimodular forms.

We fix a nonnegative integer $m$, and denote by $\mF_m [X]$ the space of
polynomials over the ring $\mF$ of holomorphic functions on $\mH$ of degree
at most $m$.  Given $\gl \in \bZ$ and a polynomial
\[
\Phi (z,X) = \sum^m_{r=0} \phi_r (z) X^r \in \mF_m [X] ,
\]
we set
\begin{equation} \label{E:us}
(\Phi \mid^X_\gl \g) (z, X) = \sum^m_{r=0} (\phi_r \mid_{\gl +2r}
\g) (z) X^r ,
\end{equation}
\[
(\Phi \md_\gl \g) (z, X) = \fJ (\g, z)^{-\gl} \Phi (\g z, \fJ (\g,
z)^2 ( X - \fK (\g, z)))
\]
for all $z \in \mH$ and $\g \in SL (2, \bR)$.   Then it can be shown that
the two operations $\mid^X_\gl$ and $\md_\gl$ determine right actions
of $SL(2, \bR)$ on $\mF_m [X]$.

\begin{dfn} \label{D:ns}
(i) An element $F (z,X) \in \mF_m [X]$ is a {\em modular
polynomial for
  $\G$ of weight $\gl$ and degree at most $m$\/} if it satisfies
\[
F \mid^X_\gl \g = F
\]
for all $\g \in \G$.

(ii) An element $\Phi (z,X) \in \mF_m [X]$ is a {\em quasimodular
polynomial for $\G$ of weight $\gl$ and degree at most $m$\/} if
it satisfies
\[
\Phi \md_\gl \g = \Phi
\]
for all $\g \in \G$.

(iii) An element $f \in \mF$ is a {\em quasimodular form for $\G$
of weight $\gl$ and depth at most $m$\/} if there are functions
$f_0, \ldots, f_m \in \mF$ such that
\begin{equation} \label{E:qq1}
(f \mid_\gl \g) (z) = \sum^m_{r=0} f_r (z) \fK (\g, z)^r
\end{equation}
for all $z \in \mH$ and $\g \in \G$.
\end{dfn}

We denote by $MP^m_\gl (\G)$ and $QP^m_\gl (\G)$ the spaces of
modular and quasimodular, respectively, polynomials for $\G$ of
weight $\gl$ and degree at most $m$.  We also denote by $QM^m_\gl
(\G)$ the space of quasimodular forms for $\G$ of weight $\gl$ and
depth at most $m$.

\begin{rem} \label{R:8n}
(i) It follows from \eqref{E:us} and Definition \ref{D:ns}(i) that
a polynomial $F(z,X)=\sum_{r=0}^m f_r (z) X^r$ is a modular polynomial
belonging to $MP_{\gl}^m (\G)$ if and only if
\[ f_r(z)\in M_{\gl +2r} (\G) \]
for each $r \in \{0, 1, \ldots, m\}$.

(ii) If $\g \in \G$ is the identity matrix in \eqref{E:qq1}, then
$\fK (\g, z) =0$.  Thus, if $f \in QM^m_\gl (\G)$ satisfies \eqref{E:qq1},
we see that
\begin{equation} \label{E:f0}
f = f_0 .
\end{equation}
On the other hand, if $m=0$, the relation \eqref{E:qq1} can be
written in the form
\[ f \mid_\gl \g = f_0 = f ;\]
hence $QM^0_\gl (\G)$ coincides with the space $M_\gl (\G)$
of modular forms of weight $\gl$ for $\G$.

(iii) If \eqref{E:qq1} is satisfied for another set of functions
$\wh{f}_0, \ldots, \wh{f}_m \in \mF$, then we have
\[ \sum^m_{r=0} (\wh{f}_r (z) -f_r (z)) \fK (\g, z)^r = 0 \]
for all $\g$ belonging to the infinite set $\G$; hence it follows
that $\wh{f}_r =f_r$ for each $r$.  Thus we see that the
quasimodular form $f$ determines the associated functions $f_0,
\ldots, f_m \in \mF$ uniquely.
\end{rem}

Given $\gl, m \in \bZ$ with $\gl > 2m \geq 0$ and a polynomial $F
(z, X) = \sum^m_{r=0} f_r (z) X^r \in \mF_m [X]$, we set
\begin{equation} \label{E:xw}
(\gL^m_\gl F) (z, X) = \sum^m_{r=0} f^\gL_r (z) X^r,
\end{equation}
\begin{equation} \label{E:j7}
(\Xi^m_\gl F) (z, X) = \sum^m_{r=0} f^\Xi_r (z) X^r,
\end{equation}
where
\[ f^\gL_k = \frac 1{k!} \sum^{m-k}_{r=0} \frac {1} {r! (\gl
-2k -r-1)!} f^{(r)}_{m -k -r} ,\]
\[ f^\Xi_k = (\gl + 2k -2m -1) \sum^{k}_{r=0} \frac {(-1)^r} {r!}
(m-k +r)! (2k +\gl -2m -r-2)! f^{(r)}_{m-k +r} \]%
for $0 \leq k \leq m$.

\begin{thm}\label{MtQ}
The complex linear endomorphisms $\gL^m_\gl$ and $\Xi^m_\gl$ of
$\mF_m [X]$ defined by \eqref{E:xw} and \eqref{E:j7} induce
isomorphisms
\begin{equation} \label{E:nw}
\gL^m_\gl: MP^m_{\gl-2m} (\G) \to QP^m_\gl (\G), \quad \Xi^m_\gl:
QP^m_\gl (\G) \to MP^m_{\gl-2m} (\G)
\end{equation}
with
\[ (\Xi^m_\gl)^{-1} = \gL^m_\gl \]
for each $\gl > 2m$.
\end{thm}

\begin{proof}
See \cite{L09b}.
\end{proof}

If $f(z) \in QM^m_\gl (\G)$ is a quasimodular form satisfying
\eqref{E:qq1}, we define the corresponding polynomial $(\mQ_\gl^m
f) (z,X) \in \mF_m [X]$ by
\begin{equation} \label{E:tp}
(\mQ_\gl^m f) (z,X) = \sum^m_{r=0} f_r (z) X^r
\end{equation}
for all $z \in \mH$, which is well-defined by Remark
\ref{R:8n}(iii).  Then it is known (see \cite{L09b}) that the resulting
complex linear map
\begin{equation} \label{E:sp}
\mQ_\gl^m: QM^m_\gl (\G) \to QP^m_{\gl} (\G)
\end{equation}
is an isomorphism whose inverse is given by
\[ (\mQ_\gl^m)^{-1} F (z,X) = F (z,0) \]
for $F(z,X) \in QP^m_{\gl} (\G)$.

\bigskip

\section{\bf{Weak Jacobi forms}}

Weak Jacobi forms generalize usual Jacobi forms (cf.\
\cite{EZ85}), and they appear, for example, in the theory of
elliptic genera and quantum field theory (see e.g.\ \cite{Wi87}).
In this section we describe connections between vector-valued
modular forms and weak Jacobi forms for the modular group $\G (1)
= SL(2, \bZ)$.

\begin{dfn}
Given integers $k$ and $\ell$, a holomorphic function $\phi: \mH
\times \bC \to \bC$ is a {\em weak modular form of weight $k$ and
index $\ell$ for $\G(1)$\/} if it satisfies
\[ \phi (\g z, \fJ (\g,z)^{-1} w) = \fJ (\g,z)^k e^{2 \pi i \ell
\fK (\g,z) w^2} \phi (z, w) ,\]%
\[ \phi (z, w + \mu z + \nu) = (-1)^{2\ell (\mu +\nu)} e^{-2
\pi i (\mu^2 z +2\mu w)} \phi (z, w) \]%
for all $(z,w) \in \mH \times \bC$, $\g \in \G$ and $\mu, \nu \in
\bZ$, and it has a Fourier expansion of the form
\[ \phi (z, \gz) = \sum_{m \geq 0} \sum_{r \in \ell + \bZ} c(m,r)
q^m {\gz}^r \]%
with $q = e^{2\pi i z}$ and $\gz = e^{2\pi i w}$.  We denote by
$\wt{J}_{k,\ell} (\G(1))$ the space of weak modular forms of
weight $k$ and index $\ell$ for $\G(1)$.
\end{dfn}

In \cite{EZ85} Eichler and Zagier discussed various properties of
the two Jacobi forms
\[ \phi_{10,1}(z,w) = \frac 1{144} (E_6(z) E_{4,1}(z,w) - E_4(z)
E_{6,1}(z,w)), \]
\[ \phi_{12,1}(z,w) = \frac 1{144} (E_4^2(z) E_{4,1}(z,w) - E_6(z)
E_{6,1}(z,w)) \]
of index 1 of weight 10 and 12.  Here $E_4, E_6$ are the
usual elliptic Eisenstein series given by
\[ E_4(z):=1+240\sum_{n\geq 1}\sigma_{3}(n)q^n , \quad
E_6(z):=1-504\sum_{n\geq 1} \sigma_5(n)q^n \]
for $z \in \mH$ (see \cite{Za94}), and $E_{4,1}, E_{6,1}$ are Jacobi
Eisenstein series of weights $4$ and $6$, respectively, with index $1$ of
the form
\[ E_{4,1} (z,w) =1+(\gz+56\gz+126+56\gz^{-1}+\gz^{-2})q+(126\gz^2+..+126
\ga^{-2})q^2+ \cdots ,\]
\[ E_{6,1} (z,w)  =1 +(\gz^2-88\gz-330-88\gz^{-1} +\gz^{-2})q+(-330\gz^2-
\cdots -330\gz^{-2}) q+ \cdots \]
for $(z,w) \in \mH \times \bC$ (cf.\ \cite[p.~23]{EZ85}).  If we set
\begin{equation} \label{E:ms}
M_* = M_* (\G (1)) = \bigoplus_{k\geq 0} M_k(\G(1)) \cong \bC [E_4,
E_6] ,
\end{equation}
the functions
\[ \wt{\phi}_{-2,1} (z,w) =\frac{\phi_{10,1} (z,w)} {\gD (z)}, \quad
\wt{\phi}_{0,1} =\frac{\phi_{12,1} (z,w)}{\gD (z)} \]
with $\gD (z)$ as in \eqref{E:j4} satisfy the following property.

\begin{thm} The ring $\tilde{J}_{ev, *}(\Gamma(1))$  of weak Jacobi
forms of even weight is a polynomial algebra over the ring $M_{*}$ on two generators $\wt{\phi}_{-2,1}$ and
$\wt{\phi}_{0,1}$.
\end{thm}

\begin{proof}
See \cite{EZ85}.
\end{proof}

The same functions can be used to obtain a correspondence between vector-valued modular forms and weak modular forms as is stated below.

\begin{thm}\label{weak1}
The map $\Psi_{n,k} : \wh{M}_{k}^{n+1} (\G(1), \rho_n) \to
\wt{J}_{k-n,n} (\G(1))$ defined by
\begin{equation} \label{E:y4}
 \Psi_{n,k} (F)=\sum_{\ell=0}^n \mmV_{k, n, \ell}^{-1}(F)
\cdot \wt{\phi}_{-2,1}^{\ell} \cdot \wt{ \phi}_{0,1}^{n-\ell}
\end{equation}
for any $F \in \wh{M}_{k}^{n+1} (\G(1), \rho_n)$ is an
isomorphism. Here, $\mmV_{k, n, \ell}^{-1}$ is the inverse of the
isomorphism in (\ref{E:wa}).
\end{thm}

\begin{proof}
It is known that the map
\[ P^n_{k-n}: \bigoplus^n_{\ell =0} M_{k-n +2\ell} (\G (1)) \to
\wt{J}_{k-n,n} (\G(1)) \]%
given by
\[ P^n_{k-n} (f_0, f_1, \ldots, f_n) = \sum^n_{\ell =0} f_\ell\;
\wt{\phi}_{-2,1}^{\ell} \wt{\phi}_{0,1}^{n -\ell} \] with $f_\ell
\in M_{k-n+2\ell} (\G (1))$ for $0 \leq \ell \leq n$ is an
isomorphism (see \cite[p.108]{EZ85}).  Thus the theorem follows
from this and the isomorphism \eqref{E:m} in Theorem \ref{main}.
\end{proof}


The above results suggest the existence of certain isomorphisms included in the next theorem.

\begin{thm}
There are canonical isomorphisms of the form
\begin{align*}
\bigoplus_{k\equiv n \, ({\rm mod} \, 2)} \wh M_{k}^{n+1}(\G
(1), \rho_n) &\cong \bigoplus_{k\equiv n \, ({\rm mod} \, 2)} \bigoplus_{\ell=0}^n M_{k-n+2\ell}(\G(1))\\
&\cong \bigoplus_{k\equiv n \, ({\rm mod} \, 2)}
\tilde{J}_{k-n,n}(\G(1)) \cong M_*[\wt{\phi}_{0,1},
\wt{\phi}_{2,1} ] ,
\end{align*}
where $M^*$ is as in \eqref{E:ms}.
\end{thm}


\bigskip

\section{\bf{Cohen-Kuznetsov liftings}}

We denote by $\mF [[X]]$ the complex algebra of formal power
series in $X$ with coefficients in $\mF$.  We fix a positive
integer $n$ and denote by $\wh{\mF}^{n+1}$ the space of $\mbb
C^{n+1}$-valued holomorphic functions on $\mcal H$.  Then the
space $\wh{\mF}^{n+1} [[X]]$ of formal power series in $X$ with
coefficients in $\wh{\mF}^{n+1}$ has the structure of a module over $\mF
[[X]]$. Note that, unlike $\mF [[X]]$, $\wh{\mF}^{n+1} [[X]]$ does not
have a ring structure. Let $\G$ be a discrete subgroup of
$SL(2,\mbb \mF)$, and let $\rho: \G \to GL(n+1, \mbb C)$ be a
representation of $\G$ in $\mbb C^{n+1}$. Then the associated
action of $\G$ on the coefficients induces an action of $\G$ on
$\wh{\mF}^{n+1} [[X]]$.

\begin{dfn} \label{D:wk}
(i) Given an integer $\gl$, a {\em Jacobi-like form of weight
$\gl$ for $\G$\/} is a formal power series $\Phi (z,X) \in \mF
[[X]]$ satisfying
\[
\Phi (\g z, \fJ (\g,z)^{-2} X) = \fJ (\g,z)^\gl e^{\fK (\g,z) X}
\Phi (z,X)
\]
for all $z \in \mcal H$ and $\g \in \G$.

(ii) A {\em vector-valued Jacobi-like form of weight $\gl$ for
$\G$ with respect to $\rho$\/} is a formal power series $\wh{\Phi}
(z,X) \in \wh{\mF}^{n+1} [[X]]$ satisfying
\begin{equation} \label{E:w1}
\wh{\Phi} (\g z, \fJ (\g,z)^{-2} X) = \fJ (\g,z)^\gl e^{\fK (\g,z)
X} \rho(\g) \wh{\Phi} (z,X)
\end{equation}
for all $z \in \mcal H$ and $\g \in \G$.
\end{dfn}

We denote by $\mJ_{\gl} (\G)$ and $\wJ_{\gl} (\G, \rho)$ the spaces
of Jacobi-like forms and vector-valued Jacobi-like forms with
respect to $\rho$, respectively, of weight $\gl$ for $\G$.

\begin{lem}\label{ind}
If $\wh{v}_{n} \in \wh{M}_{-n}^{n+1}(\G, \rho_n)$ is as in
\eqref{E:mv}, the associated formal power series given by
\begin{equation} \label{E:uj}
\wt{\Phi}_{\wv_{n}} (z,X) = \sum_{j=0}^{\infty} \frac{(-1)^j
(n-j)! D^{j} (\wv_n)} {j!} X^{j}
\end{equation}
is a vector-valued Jacobi-like form belonging to $\wJ_{-n} (\G,
\rho_n)$.
\end{lem}

\begin{proof}
Given $\ga \in SL (2, \bR)$, it can be shown by induction that
\begin{equation} \label{E:bv}
((D^{\nu} (\wv_n)) \mid_{-n +2\nu} \ga) (z) = \sum_{\ell=0}^{\nu}
\frac {(-1)^{\nu -\ell} \nu! (n-\ell)!} {\ell! (\nu -\ell)!
(n-\nu)!} \fK (\g,z)^{\nu -\ell} D^\ell (\wv_n \mid_{-n} \ga) (z)
\end{equation}
for each $\nu \geq 0$ and $z \in \mH$.  If $\g \in \G$ and if
$\wt{\Phi}_{\wv_{n}} (z,X)$ is the formal power series given by
\eqref{E:uj}, we have
\begin{align*}
\wt{\Phi}_{\wv_{n}} (\g z, \fJ (\g,z)^{-2} X) &= \sum^\infty_{j=0}
\frac {(-1)^j (n-j)!} {j!} (D^j \wv_n) (\g z) \fJ (\g,z)^{-2j} X^j\\
&= \sum^\infty_{j=0} \frac {(-1)^j (n-j)!} {j!} ((D^j \wv_n)
\mid_{-n +2j} \g) (z) \fJ (\g,z)^{-n} X^j\\
&= \sum^\infty_{j=0} \frac {(-1)^j (n-j)!} {j!} \fJ (\g,z)^{-n}
X^j\\
&\hspace{.5in} \times \sum^j_{\ell=0} \frac {(-1)^{j-\ell} j!
(n-\ell)!} {\ell! (j-\ell)! (n-j)!} \fK (\g,z)^{j-\ell} D^\ell
(\wv_n \mid_{-n} \g) (z)\\
&= \fJ (\g,z)^{-n} \sum^\infty_{j=0} \sum^j_{\ell=0} \frac
{(-1)^{\ell} (n-\ell)!} {\ell! (j-\ell)!} \fK (\g,z)^{j-\ell}
(D^\ell \wv_n) (z) X^j ,
\end{align*}
where we used the relation $\wv_n \mid_{-n} \g = \wv_n$.  On the other
hand, we see that
\begin{align*}
\fJ (\g,z)^{-n} e^{\fK (\g,z) X} \wt{\Phi}_{\wv_n} (z, X) &= \fJ
(\g,z)^{-n} \sum_{\ell=0}^{\infty} \sum_{j=0}^{\infty}
\frac{(-1)^j (n-j)! (D^{j} \wv_n) (z)} {j! \ell!} \fK (\g,z)^\ell
X^{j
+\ell}\\
&= \fJ (\g,z)^{-n} \sum_{p=0}^{\infty} \sum_{j=0}^{p} \frac{(-1)^j
(n-j)! (D^{j} \wv_n) (z)} {j! (p-j)!} \fK (\g,z)^{p-j} X^{p}.
\end{align*}
Hence we obtain
\begin{equation} \label{E:hx}
\wt{\Phi}_{\wv_{n}}(\g z, \fJ (\g,z)^{-2} X)= \fJ (\g,z)^{-n}
e^{\fK (\g,z) X} \rho_n(\g) \wt{\Phi}_{\wv_n} (z, X) ,
\end{equation}
and therefore the lemma follows.
\end{proof}

The vector-valued Jacobi-like form $\wt{\Phi}_{\wv_{n}} (z,X) \in \wJ_{-n}
(\G, \rho_n)$ may be regarded as the vector-valued version of the
Cohen-Kuznetsov lifting (cf.\ \cite{CM97, Za94}) of the vector-valued
modular form $\wv_n \in \wh{M}^{n+1}_{-n} (\G, \rho_n)$.

\bigskip

\section{\bf{Proof of main theorems}} \label{S:pm}

\noindent%
\textbf{Proof of Theorem \ref{main}}\, We consider an element
\[ \wh{g} = (g_0, g_1, \ldots, g_n) \in \bigoplus^n_{\ell =0} M_{k-n
+2\ell} (\G) \]%
with $g_\ell \in M_{k -n +2\ell}$ for each $\ell \in \{ 0, 1,
\ldots, n\}$.  Then the Cohen-Kuznetsov lifting of the modular
form $g_\ell$ given by
\begin{equation} \label{E:5x}
\wt{g}_\ell (z,X) = \sum^\infty_{j=0} \frac {g^{(j)}_\ell (z)} {j!
(j +k -n + 2\ell -1)!} X^j
\end{equation}
is a Jacobi-like form belonging to $\mJ_{k -n +2\ell} (\G)$ and
therefore satisfies
\begin{equation} \label{E:d8}
\wt{g}_\ell (\g z, \fJ (\g, z)^{-2} X) = \fJ (\g,z)^{k-n +2\ell} e^{\fK
(\g, z) X} \wt{g}_\ell (z,X)
\end{equation}
for all $\g \in \G$ (see e.g. \cite{CM97, Za94}). Similarly, by
Lemma \ref{ind} the vector-valued version of the Cohen-Kuznetsov
lifting of $\wv_n \in \wh{M}_{-n} (\G, \rho_n)$ given by
\begin{equation} \label{E:6x}
\wh{\Phi}_{\wv_n} (z,X) =\sum^\infty_{j=0} \frac {(-1)^j (n-j)!
\wv^{(j)}_n (z)} {j!} X^j ,
\end{equation}
satisfies
\begin{equation} \label{E:d9}
\wh{\Phi}_{\wh{v}_n} (\g z, \fJ (\g, z)^{-2} X)= \fJ (\g, z)^{-n}
e^{\fK (\g, z) X} \rho_n (\g) \wh{\Phi}_{\wh{v}_n} (z,X)
\end{equation}
for all $\g \in \G$.  We now set
\[ \wh{F} (z,X) = \wt{g}_\ell (z,-X) \wh{\Phi}_{\wv_n} (z,X) .\]
Then, using \eqref{E:d8} and \eqref{E:d9}, we have
\begin{align} \label{E:e7}
\wh{F} (\g z, \fJ(\g, z)^{-2} X) &= \fJ (\g, z)^{k-n +2\ell} e^{-
\fK (\g,z) X} \wt{g}_\ell (z,-X)\\
&\hspace{.6in} \times \fJ (\g, z)^{-n} e^{\fK (\g,z)
X} \rho_n (\g) \wh{\Phi}_{\wv_n} (z,X) \notag\\
&= \fJ (\g, z)^{k-2n +2\ell} \rho_n (\g) \wh{F} (z,X) \notag
\end{align}
for all $\g \in \G$.  Hence, if we write
\[ \wh{F} (z,X) = \sum^\infty_{j=0} \wh{\eta}^{g_\ell}_j (z) X^j ,\]
from \eqref{E:e7} we see that
\begin{equation} \label{E:pf1}
\wh{\eta}^{\phi_\ell}_j \in \wh{M}^{n }_{k-2n +2\ell +2j} (\G,
\rho_n)
\end{equation}
for each $j \geq 0$.
On the other hand, using \eqref{E:5x} and \eqref{E:6x}, we obtain
\begin{align*}
\wh{F} (z,X) &= \wt{g}_\ell (z,-X) \wh{\Phi}_{\wv_n} (z,X)\\
&= \sum^\infty_{r=0} \sum^\infty_{s=0} \frac {(-1)^{r +s} (n-s)!
g^{(r)}_\ell (z) \wv^{(s)}_n (z)} {r! s! (r +k -n + 2\ell -1)!}
X^{r +s}\\
&= \sum^\infty_{j=0} \sum^j_{r=0} \frac {(-1)^{j} (n-j +r)!
g^{(r)}_\ell (z) \wv^{(j-r)}_n (z)} {r! (j-r)! (r +k -n + 2\ell
-1)!} X^{j};
\end{align*}
hence we have
\[ \wh{\eta}^{ g_\ell}_j = \sum^j_{r=0} \frac {(-1)^{j} (n-j
+r)! g^{(r)}_\ell \wv^{(j-r)}_n} {r! (j-r)! (r +k -n +
2\ell -1)!} \]%
for all $j \geq 0$.  Furthermore, using \eqref{E:he}, for $0 \leq \ell \leq
n$ and $j \geq 0$ we have
\begin{align*}
[g_\ell, \wh{v}_n]^{(k-n +2\ell, -n)}_{j} &= \sum^j_{r=0} (-1)^r
\bi {k-n +2\ell +j -1} {j-r} \bi {-n+j -1} {r} g^{(r)}_\ell
\wv^{(j-r)}_n\\
&= \sum^j_{r=0} (-1)^r \frac {(k-n +2\ell +j -1)!} {(j-r)! (k-n
+2\ell +r -1)!}\\
&\hspace{.9in} \times \frac {(-n+j -1) \cdots (-n+j -r)} {r!}
g^{(r)}_\ell \wv^{(j-r)}_n\\
&= \sum^j_{r=0} \frac {(k-n +2\ell +j -1)! (n+r-j)!} {(j-r)! (k-n
+2\ell +r -1)! r! (n-j)!} g^{(r)}_\ell \wv^{(j-r)}_n\\
&= (-1)^j \frac {(k-n +2\ell +j -1)!} {(n-j)!} \wh{\eta}^{
g_\ell}_j ,
\end{align*}
which belongs to $\wh{M}^{n+1}_{k-2n +2\ell +2j} (\G, \rho_n)$ by
\eqref{E:pf1}.  Hence it follows that
\[ [g_\ell, \wh{v}_n]^{(k-n +2\ell, -n)}_{n-\ell} \in
\wh{M}^{n+1}_{k} (\G, \rho_n) ,\]%
which proves (i). We now consider a vector-valued modular form $F
\in \wh{M}_k^n (\G, \rho_n)$, and assume that
\[ L_n(z)^{-1} F (z) = {^t}(f_0 (z), f_1 (z), \ldots, f_n (z)) \]%
for all $z \in \mH$ with $f_0, \ldots, f_n \in \mF$.  Let $t$ be the first
positive integer such that $f_t$ is not identically zero.
Then by Lemma \ref{KS-2} the function $f_t$ is a modular form
belonging to $M_{k-n+2t} (\G)$.  Here, we note that $k-n+2t>0 $
because otherwise $M_{k-n+2t}(\G)=\{0\}$.  Using Lemma \ref{KS-1},
we see that the first nonzero entry of the vector
\[ L_n (z)^{-1} \mmV_{k,n,t} (f_t) (z) = L_n (z)^{-1} [f_t (z),
\wv_n (z)]^{(k-n +2t,-n)}_{n-t} \]%
is the $t$-th component which is equal to
\[ (n-t)! \ga_{k-n +2t,0} f_t (z) .\]
Thus, if we set
\[ \xi_{0,t} = \frac 1{(n-t)! \ga_{k-n +2t,0}} \mmV_{k,n,t} (f_t)
\in {\rm Im}\, \mmV_{k,n,t} = \wh{\mG}^n_{k-n+2t} (\G, \rho_n)  ,\]%
the first $t+1$ entries of the vector
\[ L_n (z)^{-1} (F(z) - \xi_{0,t} (z)) \]
are zero with $F - \xi_{0,t} \in \wh{M}^{n+1}_k (\G, \rho_n)$.
Applying the same argument to the the vector-valued modular form
$F-\xi_{0,t}$, we can find an element $\xi_{1, t+1} \in
\wh{\mG}^n_{k-n+2t +2} (\G, \rho_n)$ such that the first $t+1$
components of the vector
\[ L_n (z)^{-1} (F(z) - \xi_{0,t} (z)- \xi_{1,t+1} (z)) \]%
are all zero.  By repeating this process we obtain the expression
\[ F = \sum^{n+2 -t}_{\ell =0} \xi_{\ell, t +\ell} \in
\bigoplus^{n+2 -t}_{\ell =0} \wh{\mG}^n_{k-n+2t +2\ell} (\G,
\rho_n) ,\]%
which implies (ii).  In order to prove (iii), given $\ell \in \{
0, 1, \ldots, n \}$, we assume that
\[ \mmV_{k, n, \ell} (g) = \mmV_{k, n, \ell} (h) \]
with $g,h \in M_{k-n -2\ell} (\G)$.  Then from \eqref{E:wa} and
Lemma \ref{KS-1} we see that
\[ 0 = [g-h, \wv_n]_{n-\ell}^{(k-n +2\ell, -n)} = L_n (z) \cdot
{^t} (f_0 (z), f_1 (z), \ldots, f_n (z)) \]%
for all $z \in \mH$, where
\[ f_j = \begin{cases}
0 & \text{for $0 \leq j \leq \ell-1$;}\\
(n-j)! \ga_{k -n +2\ell, j -\ell} D^{j -\ell} (g-h) & \text{for
$\ell \leq j \leq n$.}
\end{cases}\]%
Thus we have
\[ 0 = f_\ell = (n-\ell)! \ga_{k -n +2\ell, 0} (g-h) ,\]
and therefore $g=h$, which show that $\mmV_{k, n, \ell}$ is
injective; hence it follows that the map
\[ \mmV_{k, n, \ell}: M_{k-n+2\ell}(\G) \to
\wh{\mG}_{k-n+2\ell}^{n+1} (\G,\rho_n) = {\rm Im}\,
(\mmV_{k, n, \ell}) \]%
is an isomorphism verifying (iii).

\bigskip

\noindent%
\textbf{Proof of Theorem \ref{main2}}\, Given integers $k$ and $n$
with $n \geq 0$, using the bilinear map \eqref{E:bi} and the fact that
$\ww_n \in \wh{M}^{n+1}_{-n} (\G, \rho^*_n)$ (see Proposition \ref{syme}),
we see easily that the map $\mW_{k,n}$ is well-defined; hence the theorem
follows.

\bigskip

\noindent%
\textbf{Proof of Theorem \ref{main00}}\,

 Given $k \in \bZ$, by
Theorem \ref{main}(iii) there is a canonical isomorphism
\[ \wh{M}_k^{n} (\G,\rho_n) \cong \bigoplus^n_{\ell =0}
M_{k-n +2\ell} (\G) .\]%
However, using Remark \ref{R:8n}(i), we can identify the direct
sum $\bigoplus^n_{\ell =0} M_{k-n +2\ell}(\G)$ with the space
$MP^n_{k-n} (\G)$ of modular polynomials.  Since
\[ MP^n_{k-n} (\G) \cong QP^n_{k+n} (\G) \]
by \eqref{E:nw}, we obtain the desired isomorphism. By combining
\eqref{E:m} with the isomorphism \eqref{E:sp} we also obtain a
correspondence
\[ \wh{M}_k^{n+1} (\G, \rho_n) \cong QM^n_{k+n} (\G) \]
between vector-valued modular forms and quasimodular forms.

\bigskip

\noindent%
\textbf{Proof of Theorem \ref{main0}}\,
Given $k \in \bZ$, we consider a quasimodular polynomial $F (z,X) \in QP_{k+n}^n
(\G)$ of the form
\[ F(z,X) = \sum^n_{r=0} f_r (z) X^r .\]
Since we already have the isomorphism \eqref{E:6w}, it suffices to show
that $\mV_n (F (z,X))$ belongs to $\wh{M}^{n+1}_{k} (\G, \rho_n)$.  First,
we note that $F(z,X)$ can be lifted to a Jacobi-like form
\[ \Phi_F (z,X) = \sum^{\infty}_{\ell=0} \phi_\ell (z) X^\ell \in \mJ_{k-n} (\G)
\]  with
\[ f_r = \frac 1{r!} \phi_{n-r} \]
for $0 \leq r \leq n$ satisfying
\begin{equation} \label{E:hy}
\Phi_F (\g z, \fJ (\g,z)^{-2} X) = \fJ (\g,z)^{k-n} e^{\fK (\g,z) X}
\Phi_F (z,X)
\end{equation}
for all $\g \in \G$ (cf. \cite{CL08}).  If $\wt{\Phi}_{\wv_{n}} (z,X) \in
\wJ_{-n} (\G, \rho_n)$ is as in \eqref{ind}, we set
\begin{align*}
\Psi (z,X) &= \Phi_F (z,-X) \wt{\Phi}_{\wv_{n}} (z,X)\\
&= \sum_{j=0}^{\infty} \sum^n_{\ell=0} \frac{(-1)^{j+\ell} (n-j)!} {j!}
D^{j} (\wv_n) (z) \phi_\ell (z) X^{j +\ell}\\
&= \sum_{r=0}^{\infty} \sum^r_{\ell=0} \frac{(-1)^{r} (n-r +\ell)!} {(r-\ell)!}
D^{r-\ell} (\wv_n) (z) \phi_\ell (z) X^{r} ,
\end{align*}
assuming that $\phi_\ell =0$ for $\ell >n$.  Thus we may write
\begin{equation} \label{E:gr}
\Psi (z) = \sum^\infty_{r=0} \psi_r (z) X^r ,
\end{equation}
where
\begin{align*}
\psi_r &= \sum^r_{\ell=0} \frac{(-1)^{r} (n-r +\ell)!} {(r-\ell)!}
D^{r-\ell} (\wv_n) \phi_\ell\\
&= \sum^r_{\ell=0} \frac{(-1)^{r} (n-r +\ell)! (n-\ell)!} {(r-\ell)!}
D^{r-\ell} (\wv_n) f_{n-\ell}\\
&= \sum^r_{\ell=0} \frac{(-1)^{r} (2n-r -\ell)! \ell!} {(r-n +\ell)!}
D^{r-n +\ell} (\wv_n) f_{\ell}
\end{align*}
for $r \geq 0$.  Using \eqref{E:hx} and \eqref{E:hy}, we see that
\[ \Psi (\g z, \fJ (\g,z)^{-2} X) = \fJ (\g,z)^{k-2n} \rho_n (\g) \Psi
(z,X) .\]
for $\g \in \G$.  From this and \eqref{E:gr} it follows that
\[ \psi_r \in \wh{M}^{n+1}_{k -2n +2r} (\G, \rho_n) \]
for each $r \geq 0$.  In particular, we obtain
\[ \psi_n  = \sum^n_{\ell=0} (-1)^{n} (n -\ell)! D^{\ell} (\wv_n)
f_{\ell} \in \wh{M}^{n+1}_{k} (\G, \rho_n) ,\]
and therefore the theorem follows from this and the identity $\psi_n (z) =
\mU_{n} (F(z,X))$.

\providecommand{\bysame}{\leavevmode\hbox
to3em{\hrulefill}\thinspace}
\providecommand{\MR}{\relax\ifhmode\unskip\space\fi MR }
\providecommand{\MRhref}[2]{%
  \href{http://www.ams.org/mathscinet-getitem?mr=#1}{#2}
} \providecommand{\href}[2]{#2}


\begin{thebibliography}{10}
\bibitem{BGHZ06}J.~Bruinier, G.~van~der~Geer, G.~Harder and  D.~Zagier, \emph{
The 1-2-3 of modular forms,}
 Lectures from the Summer School on
Modular Forms and their Applications held in Nordfjordeid, June
2004. Edited by Kristian Ranestad. Universitext. Springer-Verlag,
Berlin, 2008.

\bibitem{CL08}
Y.~Choie and M.~H. Lee, \emph{Quasimodular forms, {J}acobi-like forms, and
  pseudodifferential operators}, preprint.

\bibitem{CM97}
P.~B. Cohen, Y.~Manin, and D.~Zagier, \emph{Automorphic pseudodifferential
  operators}, Algebraic Aspects of Nonlinear Systems, Birkh\"auser, Boston,
  1997, pp.~17--47.

\bibitem{Ei57}
M.~Eichler, \emph{Eine {V}erallgemeinerung der {A}belschen
{I}ntegrals}, Math. Z. \textbf{67} (1957), 267--298.

\bibitem{EZ85}
M.~Eichler and D.~Zagier, \emph{The theory of {J}acobi forms},
Progress in
  Math., vol.~55, Birkh{\"{a}}user, Boston, 1985.

\bibitem{EO1}A.~Eskin and A.~Okounkov,  Asymptotics of numbers of branched
coverings of a torus and volumes of moduli spaces of holomorphic
differentials. Invent. Math. \textbf{145} (2001), 59--103

 \bibitem{EO2}
 A.~Eskin and A.~Okounkov, Andrei Pillowcases and quasimodular forms.
 Algebraic geometry and number theory, Progr. Math.,
 253, Birkhauser Boston, Boston, 2006, pp.~1--25

\bibitem{KZ95}
M.~Kaneko and D.~Zagier, \emph{A generalized {J}acobi theta
function and
  quasimodular forms}, Progress in Math., vol. 129, Birkh{\"{a}}user, Boston,
  1995, pp.~165--172.

\bibitem{KS60}
M.~Kuga and G.~Shimura, \emph{On vector differential forms attached to
  automorphic forms}, J. Math. Soc. Japan \textbf{12} (1960), 258--270.

\bibitem{LR} S.~Lelievre and  E.~Royer,   Orbitwise countings in $
H(2)$ and quasimodular forms. Int. Math. Res. Not. 2006, Art. ID
42151, 30 pp


\bibitem{L07x}
M.~H. Lee, \emph{Vector-valued {J}acobi-like forms}, Monatsh. Math.
  \textbf{152} (2007), 321--336.

\bibitem{L09b}
M.~H. Lee, \emph{Quasimodular forms and {P}oincar\'e series}, Acta Arith.
  \textbf{137} (2009), 155--169.

\bibitem{Sh59}
G.~Shimura, \emph{Sur les int{\'{e}}grales attach{\'{e}}s aux formes
  automorphes}, J. Math. Soc. Japan \textbf{11} (1959), 291--311.

\bibitem{Wi87}
E.~Witten, \emph{Elliptic genera and quantum field theory},
Comm. Math. Phys. \textbf{109} (1987), 525--536.

\bibitem{Za94}
D.~Zagier, \emph{Modular forms and differential operators}, Proc. Indian
Acad. Sci. Math. Sci. \textbf{104} (1994), 57--75.

\end{thebibliography}
\end{document}